\documentclass[a4paper,twoside,11pt]{article}
\usepackage{amsmath,latexsym,amssymb,amsfonts,amsbsy,amsthm,mathrsfs}
\usepackage{algorithm}
\usepackage{algorithmic}
\usepackage{cite}
\usepackage{array}
\usepackage{cleveref}
\usepackage{float}

\usepackage{indentfirst}
\usepackage[labelfont=bf, labelsep=period, figurename=Fig.]{caption}

\usepackage{graphicx}
\usepackage{subfigure}
\usepackage[top=1in, bottom=1in, left=1.25in, right=1.25in]{geometry}
\usepackage{epstopdf}
\usepackage{diagbox}
\usepackage{enumerate}
\usepackage{color}
\newfont{\bb}{msbm10}

\newtheorem{theorem}{Theorem}[section]
\newtheorem{definition}{Definition}[section]
\newtheorem{lemma}{Lemma}[section]

\theoremstyle{definition}
\newtheorem{example}{Example}[section]

\numberwithin{equation}{section}

 \newtheorem{prop}{Proposition}[section]

\title{Residual-based Kaczmarz methods for tensor linear equations with t-product}

\author{
Li-Lin Ji\\
School of Mathematical Sciences, Tongji University, \\
Shanghai 200092, China \\
Email: 2251237@tongji.edu.cn,\\
Juanjuan Sun \\
Key Laboratory of Intelligent Computing and 
Applications (Tongji University), 
 \\Ministry of Education,
School of Mathematical Sciences, Tongji University, \\
Shanghai 200092, China \\
Email: sunjuan@tongji.edu.cn,\\
Jun-Feng Yin\\
Key Laboratory of Intelligent Computing and 
Applications (Tongji University),
\\Ministry of Education,
School of Mathematical Sciences, Tongji University, \\
Shanghai 200092, China \\
Email:yinjf@tongji.edu.cn\\
}

\begin{document}
\cleardoublepage \pagestyle{myheadings}

\maketitle
\date{}

\markboth{\small L.-L. Ji, J.-J. Sun and J.-F. Yin}
{\small  Residual-based Kaczmarz methods for tensor linear equations with t-product}

\begin{abstract}
Tensor linear systems widely arise from high-dimensional data mining and computing, for instance, natural language processing and machine learning. A class of residual-based tensor Kaczmarz method is proposed for tensor linear equations with t-product. Theoretical analyses prove the convergence and give an upper bound of the convergence rate of the proposed method. Furthermore, an accelerated residual-based Kaczmarz method with heavy ball momentum is developed. Numerical experiments verify the efficiency of the proposed methods and demonstrate that they are faster than the existing tensor Kaczmarz methods.

\end{abstract}


\bigskip

\section{Introduction}
\label{sec:intro}

Consider the solution of consistent tensor linear systems
\begin{equation}
\mathcal{A} * \mathcal{X} = \mathcal{B},
\end{equation}
where \(\mathcal{A} \in \mathbb{C}^{n_1 \times n_2 \times n_3}\),
\(\mathcal{X} \in \mathbb{C}^{n_2 \times \ell \times n_3}\),
\(\mathcal{B} \in \mathbb{C}^{n_1 \times \ell \times n_3}\)
and
the operator $*$ denotes the t-product introduced by Kilmer and Martin \cite{2011KM},  which arises widely from high-dimensional data mining and computing, such as image reconstruction \cite{2021CQ,2013KBHH}, video analysis \cite{2024LLY} and medical imaging \cite{ke2026randomized}.

Kaczmarz method is a fundamental iteration method for consistent linear equations and image reconstruction \cite{2021CQ}. In recent times, randomized Kaczmarz methods are proposed and effectively applied to online computation \cite{2014NSW} and learning theory \cite{Chen2018}. A number of variants of Kaczmarz methods have been proposed and studied in the past decades, including block Kaczmarz method \cite{2024XY,2025XGY}, greedy Kaczmarz method \cite{2026LYWY}, extended Kaczmarz method \cite{2025WYYZ}, sparse Kaczmarz method \cite{2025WYZ,2025WYYZ}, Bregman Kaczmarz method \cite{2026YYJW,2026YW} and surrogate hyperplane Kaczmarz method \cite{2024DWYY,2025DQY}.

Recently, randomized Kaczmarz method was firstly proposed to solve consistent tensor linear systems with t-product in \cite{2022MM}, and the linear convergence rate in expectation was established. 
%To avoid the computation of the pseudoinverse in solving consistent tensor systems, the randomized average Kaczmarz method with t-product was developed in \cite{2022BZLWG}.
Inspired by the idea of sketch-and-project to tensor equations, Tang, Zhang and Li proposed adaptive variants and their Fourier-domain versions repectively \cite{2022TZL,2026TZL}. Moreover, variants of tensor Kaczmarz were presented for inconsistent tensor systems \cite{2024HZ,2025ALJL}, factorized tensor systems \cite{2025CHHH} and tensor recovery \cite{2021CQ,2024ZGP}, as well as the accelerated techniques, including greedy residual \cite{2024ZH,2024WLRL} and 
heavy ball momentum \cite{2024LLY}.

%On the other hand, many regularized Kaczmarz method are studied for tensor completion.
%Chen and Qin proposed a randomized regularized Kaczmarz algorithm for tensor recovery,
%which obviously improve the robustness and stability of the solution when the tensor system is under-determined or ill-posed.
%Further, Du and Sun proposed a tensor randomized extended Kaczmarz algorithm for the inconsistent tensor recovery.
%Zhang et al. established a new regularized Kaczmarz method
%by a sampling-based greedy strategy and greatly enhanced the efficiency of tensor recovery.

In this paper, a class of residual-based Kaczmarz method is proposed for tensor linear systems. Theoretical analyses prove the convergence of the proposed method, which has better convergence than the existing tensor Kaczmarz methods.
%The implementation of Fourier-transformed variant are given to
%further enhance the computational efficiency by FFT.
Furthermore, a heavy ball momentum-accelerated residual-based tensor Kaczmarz method is developed. Numerical examples are given to verify the efficiency of the proposed approaches
and demonstrate that they are faster than the existing tensor Kaczmarz methods.

This paper is organized as follows. By briefly reviewing the definition of t-product in Section \ref{sec2}, the two residual-based tensor Kaczmarz methods are established and their convergence are studied in Section \ref{sec3}. In Section \ref{sec4}, numerical experiments are presented to illustrate that the proposed methods are efficient. Finally, conclusions are drawn in Section \ref{sec5}.

\section{Preliminaries of Tensor} \label{sec2}

In this section, some fundamental definitions and facts in tensor algebra are reviewed, which were proposed in \cite{2011KM} and analyzed in detail in \cite{2013KBHH}.

Denote tensors, matrices and scalars by calligraphic capital letters, capital letters and lowercase letters. For a matrix $A$, we use $A^T$, $A^H$, $A^\dag$, $\|A\|_F$, $\|A\|_2$, and $\sigma_{min}(A)$ to denote its transpose, the conjugate transpose, the Moore-Penrose pseudoinverse, the Frobenius norm, the Euclidean norm and the minimum nonzero singular values of $A$, respectively.
For a third-order tensor
$\mathcal{A} $, we use the notation $\mathcal{A}_{i,:,:}$, $\mathcal{A}_{:,j,:}$, $\mathcal{A}_{:,:,k}$ to represent the $i$-th horizontal, $j$-th lateral,
and $k$-th frontal slices of A.
To condense notation,  $\mathcal{A}_k$ represents the $kth$ frontal slice of $\mathcal{A}$.

The definition of tensor-tensor t-product, denoted by t-product, is given as follows.
\begin{definition}[t-product \cite{2011KM}]
Let \(\mathcal{A} \in \mathbb{C}^{n_1 \times n_2 \times n_3}\)
and \(\mathcal{X} \in \mathbb{C}^{n_2 \times \ell \times n_3}\) represent two third-order tensors. The tensor-tensor t-product  of \( \mathcal{A} \) and \( \mathcal{X} \) is defined as follows:
\[
 \mathcal{A} * \mathcal{X} = \text{\rm fold}\left(\text{\rm bcirc}(\mathcal{A}) \cdot \text{\rm unfold}(\mathcal{X})\right),
\]
where
\[
\text{\rm bcirc}(\mathcal
{A}) =
\begin
{bmatrix}
A_{1} & A_{ n_3} &
\cdots & A_{ 2} \\
A_{2} & A_{1} &
\cdots & A_{3} \\
%A_{3} & A_{ 2} &
%\cdots & A_{ 4} \\
\vdots & \vdots & \ddots & \vdots \\
A_{n_3} & A_{n_3-1}  &
\cdots
 & A_{1}
\end
{bmatrix}
\in \mathbb{C}^{n_1 n_3 \times
 n_2n_3},
\]
and
\[
\text{\rm unfold}(\mathcal
{X}) =
\begin
{bmatrix}
\mathcal{X}_{ 1} \\
\mathcal{X}_{ 2} \\
\vdots \\
\mathcal
{X}_{n_3}
\end
{bmatrix}
\in \mathbb{C}^{n_2n_3 \times \ell}, \quad \text{\rm fold}(\text{\rm unfold}(\mathcal{X})) = \mathcal
{X}.
\]

\end{definition}

By the definition of the t-product, it holds that
$$
(\mathcal{A} * \mathcal{X} )_{i,:,:}=
\mathcal{A}_{i,:,:} * \mathcal{X},
$$
and
$$
\mathcal{A} * \mathcal{X} =
\sum_{j=1}^{n_2} \mathcal{A}_{:,j,:} * \mathcal{X}_{j,:,:}.
$$

Some classical linear algebra properties under the t-product, including transpose, orthogonality, inverse and range, 
are studied in detail in \cite{2011KM,2013KBHH}, which are briefly reviewed below.

% -------------------------------------------------

\begin{definition}[Conjugate Transpose]
The transpose of a tensor \( \mathcal{A} \in \mathbb{C}^{n_1 \times n_2 \times n_3} \), denoted by \( \mathcal{A}^H \in \mathbb{C}^{n_2 \times n_1 \times n_3} \), is defined as follows: each frontal slice of \( \mathcal{A} \) is conjugate transposed, and the order of the frontal slices is reversed from the second to the $ n $-th, that is to say
\[
(\mathcal{A}^H)_{:,:,1} = (\mathcal{A}_{:,:,1})^H, \quad (\mathcal{A}^H)_{:,:,2:n} = (\mathcal{A}_{:,:,n:2})^H.
\]
\end{definition}

For \(\mathcal{A} \in \mathbb{C}^{n_1 \times n_2 \times n_3}\)
and \(\mathcal{B} \in \mathbb{C}^{n_2 \times \ell \times n_3}\),
it follows that
$$
\text{\rm bcirc}(\mathcal
{A}^H) =\text{\rm bcirc}(\mathcal{A})^H,
$$
and
$$
(\mathcal{A} * \mathcal{B})^H
= \mathcal{B}^H * \mathcal{A}^H.
$$

% -------------------------------------------------

\begin{definition}[Identity tensor]
The identity tensor
\( \mathcal{I} \in \mathbb{C}^{n_1 \times n_1 \times n_3} \) is defined as a tensor where the first frontal slice is the \( n_1 \times n_1 \) identity matrix and all subsequent frontal slices are zero matrices, i.e.,
\[
\mathcal{I}_{:, :, 1} = I_{n_1}, \quad \mathcal
{I}_{:, :, 2:n} = O,
\]
where
\( I_{n_1} \) denotes the \( n_1 \times n_1 \)
 identity matrix, so that the block circulant representation of the identity tensor is
$
\text{\rm bcirc}(\mathcal
{I}) = I_{n_1n_3}.
$
\end{definition}
% -------------------------------------------------

\begin{definition}[Inverse]
A tensor \( \mathcal{A} \in \mathbb{C}^{n_1 \times n_1 \times n_2} \) is invertible if there exists a tensor \( \mathcal{B} \in \mathbb{C}^{n_1 \times n_1 \times n_2} \) satisfying 
\[
\mathcal{A} * \mathcal{B} = \mathcal{B} * \mathcal{A} = \mathcal{I},
\]
where \( \mathcal{I} \) is the identity tensor. Then, \( \mathcal{B} \) is called the inverse of \( \mathcal{A} \) under the t-product, and is denoted by \( \mathcal{A}^{-1} \).

\end{definition}

% -------------------------------------------------

\begin{definition}[Moore-Penrose inverse]
Let \( \mathcal{A} \in \mathbb{C}^{n_1 \times n_2 \times n_3} \). If there exists a unique tensor \( \mathcal{X} \in \mathbb{C}^{n_2 \times n_1 \times n_3} \) such that
\begin{equation*}
\mathcal{A} * \mathcal{X} * \mathcal{A} = \mathcal{A}, \quad
\mathcal{X} * \mathcal{A} * \mathcal{X} = \mathcal{X}, \quad
(\mathcal{A} * \mathcal{X})^H = \mathcal{A} * \mathcal{X}, \quad
(\mathcal{X} * \mathcal{A})^H = \mathcal{X} * \mathcal{A},
\end{equation*}
then \( \mathcal{X} \) is called the Moore-Penrose inverse of \( \mathcal{A} \) under the t-product, denoted by \( \mathcal{A}^\dagger \).

\end{definition}

% -------------------------------------------------

\begin{definition}[Orthogonal tensor]
A tensor \( \mathcal{P} \in \mathbb{C}^{n_1 \times n_2 \times n_3}\) is called orthogonal if it satisfies the following conditions:
\[
\mathcal{P} * \mathcal{P}^H = \mathcal{P}^H * \mathcal{P} = \mathcal{I},
\]
where \( \mathcal{I} \) is the identity tensor.
\end{definition}

% -------------------------------------------------

\begin{definition}[inner product]
The inner product between $\mathcal{A}$ and
\(\mathcal{B} \in \mathbb{C}^{n_1 \times n_2 \times n_3}\)
is defined as
\[
\langle \mathcal{A}, \mathcal{B}\rangle
=\Sigma_{i,j,k} \mathcal{A}_{i,j,k} \overline{\mathcal{B}_{i,j,k}} ,
\]
where \( \overline{\mathcal{B}_{i,j,k}}  \)  is the conjugate  of \( \mathcal{B}_{i,j,k} \).
\end{definition}

For \(\mathcal{A}\in\mathbb{C}^{n_1\times n_2\times n_3}\), \(\mathcal{B}\in\mathbb{C}^{n_2\times \ell \times n_3}\) and \(\mathcal{C}\in\mathbb{C}^{n_1\times \ell \times n_3}\), it holds that
\[\langle\mathcal{A}*\mathcal{B},\mathcal{C}\rangle=\langle\mathcal{B},\mathcal{A}^{T}*\mathcal{C}\rangle.\]

% -------------------------------------------------

\begin{definition}[spectral norm and Frobenius norm]
The spectral norm and Frobenius norm of
\(\mathcal{A} \in \mathbb{C}^{n_1 \times n_2 \times n_3}\)
are defined as
\begin{eqnarray*}
\|\mathcal{A}\|_{2} &=& \|\operatorname{bcirc}(\mathcal{A})\|_{2},\\
\|\mathcal{A}\|_{F} &=& \sqrt{\langle\mathcal{A},\mathcal{A}\rangle}
= \sqrt{\sum_{ijk}(\mathcal{A}_{ijk})^{2}},
\end{eqnarray*}denoted by 2-norm and F-norm respectively.
\end{definition}

For \(\mathcal{A}\in\mathbb{C}^{n_1\times n_2\times n_3}\) and  \(\mathcal{B}\in\mathbb{C}^{n_2\times  \ell \times n_3}\),
it holds that
%$$
%\|\mathcal{A}*\mathcal{B}\|_{F}\l%eq %\|\mathcal{A}\|_{2}\|\mathcal{B}%\$|_{F},
%$$
$$
\sigma_{\min}^2 (\text{bcirc}(\mathcal{A}))
\|\mathcal{B}\|_{F} \leq
\|\mathcal{A}*\mathcal{B}\|_{F}\leq
\sigma_{\max}^2 (\text{bcirc}(\mathcal{A}))
\|\mathcal{B}\|_{F}.
$$
% -------------------------------------------------

\begin{prop}
Let \(\mathcal{A}\in\mathbb{C}^{n_1\times n_2\times n_3}\), \(\mathcal{X}\in\mathbb{C}^{n_2\times \ell \times n_3}\) and \(\mathcal{B}\in\mathbb{C}^{n_1\times \ell \times n_3}\). 
Then, the minimum Frobenius norm solution of a consistent system $\mathcal{A}*\mathcal{X} = \mathcal{B}$
is unique, and is $\mathcal{X}=\mathcal{A}^\dag* \mathcal{B}$.
\end{prop}

In this paper, the cases within the field of real numbers are studied.

\section{Residual-based tensor Kaczmarz methods} \label{sec3}

Due to the t-product of a tensor is separable in the first order, the idea of Kaczmarz methods is naturally extended to solve the consistent tensor equation by seeking the approximate solution to minimize $ \|\mathcal{X}-\mathcal{X}_k\|_F$ with the constraint
$$
\mathcal{A}_{i_k,:,:} * \mathcal{X} = \mathcal{B}_{i_k,:,:},
$$
where the index $i_k$ could be chosen cyclically or randomly.

By randomly selecting index $i_k$ according to 
$
p(i) = \frac{\|\mathcal{A}_{i,:,:}\|_F }{ \|\mathcal{A}\|_F}
(i=1,2,\ldots,n_1),
$ the tensor Kaczmarz iteration formula is reformulated by
\begin{equation}\label{eqn:trk1}
\mathcal{X}_{k+1} = \mathcal{X}_k - \mathcal{A}_{i_k,:,:}^\dagger * (\mathcal{A}_{i_k,:,:} * \mathcal{X}_k - \mathcal{B}_{i_k,:,:}),
\end{equation}
or equivalently
\begin{equation}\label{eqn:trk2}
\mathcal{X}_{k+1} = \mathcal{X}_k -
\mathcal{A}_{i_k,:,:}^T * (\mathcal{A}_{i_k,:,:}*\mathcal{A}_{i_k,:,:}^T)^\dagger * (\mathcal{A}_{i_k,:,:} * \mathcal{X}_k - \mathcal{B}_{i_k,:,:}),
\end{equation}
which corresponds to the randomized tensor Kaczmarz method \cite{2022MM}.

Moreover, by selecting the index $i_k$ according to $$
i_k = {\arg \max}_{i \in [m]} \|\mathcal{A}_{i,:,:} * \mathcal{X} - \mathcal{B}_{i,:,:}\|_F,
$$where $\arg\max$ denotes the index $i \in [m]$ at which 
$\left\| \mathcal{A}_{i,:,:} * \mathcal{X} - \mathcal{B}_{i,:,:} \right\|_F$
attains its maximum value, the tensor Kaczmarz iteration formula is reformulated by the randomized tensor Kaczmarz-Motzkin method \cite{2025NYL}.

%The maximum residual Kaczmarz methods, was proposed in
%where the index $i_k$ corresponds to the largest residual,
%determined as

%\begin{theorem}
%Let \(\mathcal{X}_*\) denote the minimal Frobenius norm solution of \(\mathcal{A} * \mathcal{X} = \mathcal{B}\), and let \(\mathcal{X}_k\) be generated by the iterative scheme \eqref{eqn:trk} with the initial value \(\mathcal{X}_0 = 0\) and a stepsize \(\alpha \in (0, 2/\beta)\), where
%\[
%\beta = \max_{i \in [m]} \frac{\sigma_{\text{max}}^2 (\text{bcirc}(\mathcal{A}_{i,:,:}))}{\|\mathcal{A}_{i,:,:}\|_F^2}.
%\]
%At each iteration, the index \(i_k \in [m]\) is randomly selected according to the probability distribution
%\[
%\mathbb{P}(i_k = i) = \frac{\|\mathcal{A}_{i,:,:}\|_F^2}{\|\mathcal{A}\|_F^2} \quad \text{for } i = 1, \ldots, m.
%\]
%In expectation, the error at the \(k\)th iteration satisfies
%\[
%\mathbb{E} \left[ \|\mathcal{X}_k - \mathcal{X}_*\|_F^2 \right] \leq \left(1 - \alpha (2 - \alpha \beta) \frac{\sigma_{\text{min}}^2 (\text{bcirc}(\mathcal{A}))}{\|\mathcal{A}\|_F^2} \right)^k \|\mathcal{X}_0 - \mathcal{X}_*\|_F^2.
%\]
%\end{theorem}

%To avoid the computation of the  Moore-Penrose inverse of \(\mathcal{A}_{i_k,:,:}\) or the inverse of \(\mathcal{A}_{i_k,:,:} * \mathcal{A}_{i_k,:,:}^T\), Bao etc proposed a tensor randomized average block Kaczmarz method
%where the iteration formula is reformulated as:
%\begin{equation}\label{eqn:trk}
%\mathcal{X}_{k+1} = \mathcal{X}_k - \mathcal{A}_{i_k,:,:}^\dagger * (\mathcal{A}_{i_k,:,:} * \mathcal{X}_k - \mathcal{B}_{i_k,:,:}).
%\end{equation}

To avoid the computation of the  Moore-Penrose inverse of \(\mathcal{A}_{i_k,:,:}\) or \(\mathcal{A}_{i_k,:,:} * \mathcal{A}_{i_k,:,:}^T\),
 a  regularized Kaczmarz iteration for tensor recovery is given in \cite{2021CQ} as
\[
\mathcal{X}_{k+1} = \mathcal{X}_k - \frac{\alpha}{\|\mathcal{A}_{i_k,:,:}\|_F^2} \mathcal{A}_{i_k,:,:}^T * (\mathcal{A}_{i_k,:,:} * \mathcal{X}_k - \mathcal{B}_{i_k,:,:}).
\]

%Denote the tensor residual $\eta_k= \mathcal{B} - \mathcal{A} \mathcal{X}_k$,
Given a tensor $\eta_k \in\mathbb{R}^{n_1\times \ell \times n_3}$
and choosing the composite direction $\mathcal{A}^T * \eta_k$,
the new iterate scheme is defined by
\begin{equation}\label{eqn:RATK1}
\mathcal{X}_{k+1} = \mathcal{X}_k + \lambda \mathcal{A}^T * \eta_k,
\end{equation}
satisfying
\begin{equation}\label{eqn:innerproductconstraint}
 \langle \eta_k, \mathcal{A} * \mathcal{X}_{k+1} \rangle
 = \langle \eta_k, \mathcal{B} \rangle.
\end{equation}

By substituting \eqref{eqn:RATK1} into \eqref{eqn:innerproductconstraint}, it is obtained that
\begin{equation*}
 \langle \eta_k, \mathcal{A} * (\mathcal{X}_k + \lambda \mathcal{A}^T * \eta_k) \rangle
 = \langle \eta_k, \mathcal{B} \rangle,
\end{equation*}
or
  $$
 \langle \eta_k, \mathcal{A} * \mathcal{X}_k \rangle+
  \lambda \langle \eta_k, \mathcal{A} *  \mathcal{A}^T * \eta_k) \rangle
 = \langle \eta_k, \mathcal{B} \rangle,
 $$
 which gives
 $$
 \lambda = \frac{\langle \eta_k, \mathcal{B}-\mathcal{A} * \mathcal{X}_k \rangle}{\langle \mathcal{A}^T * \eta_k,   \mathcal{A}^T * \eta_k \rangle}.
  %= \frac{ \langle \eta_k, \eta_k \rangle }{\|\mathcal{A}^T * \eta_k\|_F^2}.
 $$

It is obvious that different tensor $\eta_k$ leads to different tensor Kaczmarz iteration. 
% If $\eta_k=I$, it leads to  tensor sampling Kaczmarz method.
By adaptively choosing the residual tensor as the weight
$\eta_k= \mathcal{B} - \mathcal{A}*
\mathcal{X}_k$,
the residual-based Kaczmarz method for solving tensor systems
iterates by
\begin{equation}\label{eqn:RTK}
\mathcal{X}_{k+1} = \mathcal{X}_k + \frac{ \langle \eta_k, \eta_k \rangle }{\|\mathcal{A}^T * \eta_k\|_F^2} \mathcal{A}^T * \eta_k,
\end{equation}
which is described in detail below.

\begin{algorithm}[!htbp]
	\centering
	\caption{Residual-based Kaczmarz method for solving tensor systems}\label{alg:RTK}
	\begin{algorithmic}[1]
	\STATE{\textbf{Input}: $\mathcal{A} \in \mathbb{R}^{n_1 \times n_2 \times n_3}$, $\mathcal{B}\in \mathbb{R}^{n_1 \times \ell \times n_3} $, $\mathcal{X}_{0}=0\in \mathbb{R}^{n_2 \times \ell \times n_3}$}.
		\STATE{\textbf{Output}: $\mathcal{X}_{k+1}$}.
		\FOR{$k=0,1,2,\ldots$}
		\STATE Compute:
		 $\eta_k= \mathcal{B} - \mathcal{A} *\mathcal{X}_k$,
%        \STATE Compute:
%		 $\mathcal{A}^T * \eta_k$.
		\STATE Update:
		 $\mathcal{X}_{k+1} = \mathcal{X}_k + \frac{ \langle \eta_k, \eta_k \rangle }{\|\mathcal{A}^T * \eta_k\|_F^2} \mathcal{A}^T * \eta_k$.
		\ENDFOR
	   \end{algorithmic}
    \end{algorithm}

When $n_3=1$, the residual-based tensor Kaczmarz method covers the surrogate hyperplane Kaczmarz method for solving consistent linear equations \cite{2023WY,2024DWYY}.

Before analyzing the convergence performance of the residual-based Kaczmarz method, a lemma on Kantorovich inequality is given as follows.

\begin{lemma}
Given a linear and self-adjoint operator \( B \) of the Hilbert space \( \mathcal{H} \). If the real numbers \( m, M \) and the operator \( B \) fulfill the condition
\[
0 < mE \leq B \leq ME,
\]
where \( E \) is the identity operator in \( \mathcal{H} \), then for all \( x \in \mathcal{H} \),
\[
x^T B x \, x^T B^{-1} x \leq \frac{(M + m)^2}{4mM} (x^T x)^2.
\]
\end{lemma}

Then, the convergence of the residual-based tensor Kaczmarz method is established as follows.

\begin{theorem}\label{thm:RTK}

Let \(\mathcal{X}_*\) denote the minimal Frobenius norm solution of \(\mathcal{A}* \mathcal{X} = \mathcal{B}\), and let \(\mathcal{X}_k\) be generated by the Algorithm \ref{alg:RTK} with the initial value \(\mathcal{X}_0 \in \mathrm{Range}(A^T)\).
Then, the iteration \eqref{eqn:RTK} converges and the error at the \(k\)th iteration satisfies
\[
\| \mathcal{X}_k - \mathcal{X}_* \|_F^2 \leq 
\left(1 -  q\right)^{k}
\| \mathcal{X}_0 - \mathcal{X}_* \|_F^2, 
\]
where
\[
q =
\frac{
4\sigma_{\min}^2(\mathrm{bcirc}(\mathcal{A}))
\sigma_{\max}^2(\mathrm{bcirc}(\mathcal{A}))
}{
\left(
\sigma_{\min}^2(\mathrm{bcirc}(\mathcal{A}))
+
\sigma_{\max}^2(\mathrm{bcirc}(\mathcal{A}))
\right)^2
},
\]
$\sigma_{\min}(M) $ and $\sigma_{\max}(M)$ represent the minimum singular value and the maximum singular value of the matrix $M$, respectively.
\end{theorem}

\begin{proof}

Let $\mathcal{X}_*$ be the exact solution of $ \mathcal{A}* \mathcal{X}_*= \mathcal{B}$
and
$
\eta_k = \mathcal{B} - \mathcal{A} *\mathcal{X}_k = \mathcal{A}* (\mathcal{X}_* - \mathcal{X}_k).
$
Subtracting $\mathcal{X}_*$
from both sides of \eqref{eqn:RTK}, it follows that
\begin{eqnarray*}
\mathcal{X}_{k+1} - \mathcal{X}_* &=& \mathcal{X}_k - \mathcal{X}_* + \frac{ \langle \eta_k, \eta_k \rangle }{\|\mathcal{A}^T * \eta_k\|_F^2} \mathcal{A}^T * \eta_k\\
&=& \mathcal{X}_k - \mathcal{X}_* + \frac{ \langle \mathcal{A}* (\mathcal{X}_* - \mathcal{X}_k), \mathcal{A}* (\mathcal{X}_* - \mathcal{X}_k) \rangle }{\|\mathcal{A}^T * \eta_k\|_F^2} \mathcal{A}^T * \mathcal{A}* (\mathcal{X}_* - \mathcal{X}_k)\\
   &=& (I-\mathcal{P}_k)(\mathcal{X}_k - \mathcal{X}_*),
\end{eqnarray*}
where
$\mathcal{P}_k \mathcal{Z}=\frac{ \langle (\mathcal{X}_* - \mathcal{X}_k), \mathcal{A}^T * \mathcal{A}*\mathcal{Z} \rangle }{\|\mathcal{A}^T * \eta_k\|_F^2} \mathcal{A}^T * \mathcal{A}* (\mathcal{X}_* - \mathcal{X}_k)$ is an orthogonal projection
($\mathcal{Z}$ is a tensor with compatible dimensions).

Taking the Frobenius norm of the equality and using the Pythagorean theorem for tensor orthogonal projection in \cite{2022MM}, it is obvious that
\begin{eqnarray*}
\left\| \mathcal{X}_{k+1} - \mathcal{X}_* \right\|_F^2  &=& \left\| (I-\mathcal{P}_{k} )(\mathcal{X}_k - \mathcal{X}_*) \right\|_F^2 \\
&=& \left\| \mathcal{X}_k - \mathcal{X}_* \right\|_F^2 -\left\| \mathcal{P}_k(\mathcal{X}_k - \mathcal{X}_*) \right\|_F^2.
\end{eqnarray*}

According to Lemma 3.1 of the Kantorovich inequality, it is obtained that
\begin{align*}
\left\| \mathcal{P}_k(\mathcal{X}_k - \mathcal{X}_*) \right\|_F^2
&= \big\langle \mathcal{P}_k(\mathcal{X}_k - \mathcal{X}_*),
               \mathcal{P}_k(\mathcal{X}_k - \mathcal{X}_*) \big\rangle \\[0.3em]
&= \frac{
\big\langle \mathcal{X}_k - \mathcal{X}_*,
            \mathcal{A}^T * \mathcal{A} * (\mathcal{X}_k - \mathcal{X}_*) \big\rangle^2
}{
\big\langle \mathcal{X}_k - \mathcal{X}_*,
            (\mathcal{A}^T * \mathcal{A})^2 * (\mathcal{X}_k - \mathcal{X}_*) \big\rangle
} \\[0.5em]
&\ge q\,
   \big\langle \mathcal{X}_k - \mathcal{X}_*,
               \mathcal{X}_k - \mathcal{X}_* \big\rangle.
\end{align*}
Note that $0 < q < 1$. It follows that the error at the \(k\)th iteration satisfies
\begin{eqnarray*}
\| \mathcal{X}_k - \mathcal{X}_* \|_F^2 & \leq & ( 1 -  q )
\| \mathcal{X}_{k-1} - \mathcal{X}_* \|_F^2 \\
& \leq &( 1 -  q)^{k}
\| \mathcal{X}_0 - \mathcal{X}_* \|_F^2,
\end{eqnarray*}
which completes the proof.
\end{proof}

Polyak’s momentum, which is also called the heavy ball method, was first proposed in \cite{1964Polyak}, widely used in randomized Kaczmarz methods for applications in computed tomography \cite{jin2024adaptive} and video recovery \cite{2024LLY}.

Inspired by the idea of heavy ball momentum acceleration, a residual-based Kaczmarz method with heavy ball momentum is proposed as follows
\begin{equation}\label{RTK-HB}
\mathcal{X}_{k+1} = \mathcal{X}_k + \frac{ \langle \eta_k, \eta_k \rangle }{\|\mathcal{A}^T * \eta_k\|_F^2} \mathcal{A}^T * \eta_k + \gamma_k (\mathcal{X}_k- \mathcal{X}_{k-1}).    
\end{equation}

\begin{algorithm}[!htbp]
	\centering
	\caption{Residual-based Kaczmarz method with heavy ball momentum}\label{alg:RTK-HB}
	\begin{algorithmic}[1]
	\STATE{\textbf{Input}: $\mathcal{A} \in \mathbb{R}^{n_1 \times n_2 \times n_3}$, $\mathcal{B}\in \mathbb{R}^{n_1 \times \ell \times n_3} $, $\mathcal{X}_{0}=0\in \mathbb{R}^{n_2 \times \ell \times n_3}$}.
		\STATE{\textbf{Output}: $\mathcal{X}_{k+1}$}.
		\FOR{$k=0,1,2,\ldots$}
		\STATE Compute:
		 $\eta_k= \mathcal{B} - \mathcal{A}* \mathcal{X}_k$,
%        \STATE Compute:
%		 $\mathcal{A}^T * \eta_k$.
		\STATE Update:
		 $\mathcal{X}_{k+1} = \mathcal{X}_k + \frac{ \langle \eta_k, \eta_k \rangle }{\|\mathcal{A}^T * \eta_k\|_F^2} \mathcal{A}^T * \eta_k + \gamma_k (\mathcal{X}_k- \mathcal{X}_{k-1})$.
		\ENDFOR
	   \end{algorithmic}
    \end{algorithm}

In order to choose the suitable parameter $\gamma_k$, we minimize the function
\[
G(\gamma_k)
= \frac{1}{2}
\left\|
\mathcal{X}_k
+ \alpha_k \mathcal{A}^T * \eta_k
+ \gamma_k \bigl(\mathcal{X}_k - \mathcal{X}_{k-1}\bigr)
- \mathcal{X}_*
\right\|_F^2 ,
\]
where
\[
\alpha_k
= \frac{\langle \eta_k, \eta_k \rangle}
{\|\mathcal{A}^T * \eta_k\|_F^2}.
\]

By taking the derivative of $G(\gamma_k)$ with respect to $\gamma_k$, it is obtained that
\[
\frac{dG}{d\gamma_k}
=
\alpha_k
\left\langle
\mathcal{A}^T * \eta_k,\,
\mathcal{X}_k - \mathcal{X}_{k-1}
\right\rangle
+
\gamma_k
\left\|
\mathcal{X}_k - \mathcal{X}_{k-1}
\right\|_F^2
+
\left\langle
\mathcal{X}_k - \mathcal{X}_{k-1},\,
\mathcal{X}_k - \mathcal{X}_*
\right\rangle.
\]
It is derived from $\frac{dG}{d\gamma_k} = 0$ that
\[
\gamma_k
=
\frac{
-\alpha_k
\left\langle
\mathcal{A}^T * \eta_k,\,
\mathcal{X}_k - \mathcal{X}_{k-1}
\right\rangle
-
\left\langle
\mathcal{X}_k - \mathcal{X}_{k-1},\,
\mathcal{X}_k - \mathcal{X}_*
\right\rangle
}{
\left\|
\mathcal{X}_k - \mathcal{X}_{k-1}
\right\|_F^2
}.
\]

Since
\[
\frac{d^2G}{d\gamma_k^2}
=
\|\mathcal{X}_k - \mathcal{X}_{k-1}\|_F^2 \ge 0,
\]
the function $G(\gamma_k)$ is convex in $\gamma_k$, which guarantees that the true stationary point is the global minimizer.

However, since the exact solution $\mathcal X_*$ is unknown in practical situations, 
the term $\langle \mathcal X_k - \mathcal X_{k-1},\, \mathcal X_k -\mathcal X_* \rangle$ cannot be computed explicitly. One possible approximate value of $\gamma_k$ is chosen as
\[
\gamma_k
=
\frac{
-\alpha_k
\left\langle
\mathcal{A}^T * \eta_k,\,
\mathcal{X}_k - \mathcal{X}_{k-1}
\right\rangle}{
\left\|
\mathcal{X}_k - \mathcal{X}_{k-1}
\right\|_F^2
}.
\]

Although the adopted $\gamma_k$ is approximate, it is derived from the dominant computable term and thus still preserves the descent tendency.

Before the convergence of the residual-based tensor Kaczmarz with heavy ball momentum is established, a lemma is reviewed as follows.

\begin{lemma}\label{lem:root}
Given a real-coefficient quadratic equation
\[
\lambda^{2} + a\lambda + b = 0.
\]
Then, the absolute value of two roots satisfy $|\lambda|< 1$ if and only if
\[
|b| < 1,
\]
\[
1+a+b > 0,
\]
\[
1-a+b > 0.
\]
\end{lemma}

According to the lemma above, the theorem of the convergence of the residual-based Kaczmarz method with heavy ball momentum is given below.

\begin{theorem}
Let $\mathcal{X}_*$ denote the minimal Frobenius norm solution of the
consistent tensor linear system $\mathcal{A} * \mathcal{X} = \mathcal{B},$ and let $\mathcal{X}_0  \in \mathrm{Range}(\mathcal{A}^T)$.
Let $\{\mathcal{X}_k\}$ be generated by the residual-based Kaczmarz method with heavy ball momentum, and q is defined in Theorem \ref{thm:RTK}.

Then, Algorithm \ref{alg:RTK-HB} converges linearly in expectation when $0 \le \gamma_k < \gamma_*,$ where $\gamma_* =
\frac{-(5-q) + \sqrt{(5-q)^2 + 16q}}{8}.$ More precisely, there exists $\rho \in (0,1)$ and $C>0$ such that
\[
\|\mathcal{X}_k - \mathcal{X}_*\|_F^2
\le
C\rho^k.
\]
\end{theorem}

\begin{proof}
Define the error tensor $\mathcal{E}_k = \mathcal{X}_k - \mathcal{X}_*,$ and it is obvious that
\[
\eta_k
=
\mathcal{B} - \mathcal{A} * \mathcal{X}_k
=
\mathcal{A} * (\mathcal{X}_* - \mathcal{X}_k)
=
- \mathcal{A} * \mathcal{E}_k.
\]

The iteration without momentum is shown as follows
\[
\widetilde{\mathcal{X}}_{k+1}
=
\mathcal{X}_k
+
\frac{\langle \eta_k,\eta_k\rangle}
{\|\mathcal{A}^T * \eta_k\|_F^2}
\mathcal{A}^T * \eta_k,
\]
and the corresponding error satisfies
\[
\widetilde{\mathcal{E}}_{k+1}
=
(I - \mathcal P_k)\mathcal{E}_k,
\]
where $\mathcal P_k$ is the orthogonal projection operator defined in
Theorem \ref{thm:RTK}. By Theorem \ref{thm:RTK}, it is shown that
\begin{equation}\label{eq:3.6}
\|\widetilde{\mathcal{E}}_{k+1}\|_F^2
\le
(1-q)\|\mathcal{E}_k\|_F^2.
\end{equation}

Then, the iteration with heavy ball momentum is exhibited as follows
\begin{equation*}
\mathcal{X}_{k+1}
=
\widetilde{\mathcal{X}}_{k+1}
+
\gamma_k(\mathcal{X}_k - \mathcal{X}_{k-1}),
\end{equation*}
which yields the error recursion
\begin{equation}\label{eq:3.7}
\mathcal{E}_{k+1}
=
\widetilde{\mathcal{E}}_{k+1}
+
\gamma_k(\mathcal{E}_k - \mathcal{E}_{k-1}).    
\end{equation}

Taking the Frobenius norm on both sides of \eqref{eq:3.7} and expanding the right hand side, it follows that
\[
\|\mathcal{E}_{k+1}\|_F^2
=
\|\widetilde{\mathcal{E}}_{k+1}\|_F^2
+
\gamma_k^2\|\mathcal{E}_k - \mathcal{E}_{k-1}\|_F^2
+
2\gamma_k
\langle
\widetilde{\mathcal{E}}_{k+1},
\mathcal{E}_k - \mathcal{E}_{k-1}
\rangle.
\]

By applying the Cauchy–Schwarz inequality and Young's inequality
$2ab \le a^2 + b^2$, it follows that
\[
\|\mathcal{E}_{k+1}\|_F^2
\le
(1+\gamma_k)\|\widetilde{\mathcal{E}}_{k+1}\|_F^2
+
\gamma_k(1+\gamma_k)
\|\mathcal{E}_k - \mathcal{E}_{k-1}\|_F^2.
\]
From the formula \eqref{eq:3.6} and the inequality $\|\mathcal{E}_k - \mathcal{E}_{k-1}\|_F^2
\le
2\|\mathcal{E}_k\|_F^2
+
2\|\mathcal{E}_{k-1}\|_F^2$, it gives that
\[
\|\mathcal{E}_{k+1}\|_F^2
\le
\big[(1+\gamma_k)(1-q)+2\gamma_k(1+\gamma_k)\big]
\|\mathcal{E}_k\|_F^2
+
2\gamma_k(1+\gamma_k)
\|\mathcal{E}_{k-1}\|_F^2.
\]

Denote $v_k$ by $\|\mathcal{E}_k\|_F^2$, then the sequence $\{v_k\}$ satisfies the second-order recursion
\begin{equation*}
v_{k+1}
\le
a v_k + b v_{k-1},
\end{equation*}
where $a = (1+\gamma_k)(1-q) + 2\gamma_k(1+\gamma_k),$ 
$b = 2\gamma_k(1+\gamma_k).$

Denote $\lambda_1,\lambda_2$ by the roots of the  associated characteristic equation. According to Lemma \ref{lem:root} and the definition of exponential stability, if the spectral radius $\rho := \max\{|\lambda_1|,|\lambda_2|\}$ 
satisfies $\rho < 1$, then $v_k \le C\rho^k$ 
for some constant $C>0$, which implies
\[
\|\mathcal{E}_k\|_F^2
\le
C\rho^k,
\]
when
\[
0 \le \gamma_k < \min\{\gamma_*,\frac{-1+\sqrt{3}}{2}\},
\]where $\gamma_* =
\frac{-(5-q) + \sqrt{(5-q)^2 + 16q}}{8}.$

Therefore, for $0 \le \gamma_k < \gamma_*$, the spectral radius is strictly less than $1$,
and thus the iteration \eqref{RTK-HB} achieves linear convergence in expectation.
\end{proof}

\section{Numerical Experiments} \label{sec4}

In this section, numerical experiments are presented to evaluate the convergence performance of the proposed residual-based tensor Kaczmarz method (denoted by RTK) and residual-based tensor Kaczmarz method with heavy ball momentum (denoted by RTK-HB). These methods are compared with the randomized Kaczmarz method \cite{2022MM}, the sampling Kaczmarz-Motzkin method \cite{2025NYL} and the almost-maximal residual block Kaczmarz method \cite{2024ZH} for tensor linear systems, shortened by TRK, TSKM and TBEM, respectively.

All the iterations stop when the relative solution error satisfies $\mathrm{RSE} =
\frac{\left\| \mathcal{X} - \mathcal{X}^\ast \right\|_F}
{\left\| \mathcal{X}^\ast \right\|_F}<10^{-6} $ or the number of the iterations achieves the maximal steps, e.g., 5000. The number of iterations denoted by IT and computational time in seconds denoted by CPU are compared. In image recovery, the peak signal-to-noise ratio (denoted by PSNR) in terms of the tensor \( \mathcal{X} \) can be expressed as
\[
\text{PSNR} = 10 \cdot \log_{10} \left( \frac{\| \mathcal{X}_{\text{max}} \|^2}{\text{MSE}} \right),
\]
where \( \mathcal{X}_{\text{max}} \) is the maximum pixel intensity across all channels and the mean squared error (denoted by MSE) is defined as
\[
\text{MSE} = \frac{1}{m n c} \sum_{i=1}^{m} \sum_{j=1}^{n} \sum_{k=1}^{c} \left( \mathcal{X}^*_{i,j,k} - \mathcal{X}_{i,j,k} \right)^2,
\]
where  $\mathcal{X}^*_{i,j,k},\mathcal{X}_{i,j,k} \in\ \mathbb{R}^{m\times n\times c}$ are established from the original image and the recovered image.

\begin{example}
In this experiment, the test tensors are random tensors. Both the coefficient tensor $\mathcal{A} \in \mathbb{R}^{m \times l \times n}$ and the true solution tensor $\mathcal{X}^\ast \in \mathbb{R}^{l \times p \times n}$ are randomly generated from the standard normal distribution. The corresponding right-hand side tensor $\mathcal{B} \in \mathbb{R}^{m \times p \times n}$ is obtained by computing the $t$-product of $\mathcal{A}$ and $\mathcal{X}^\ast$. If $m \ge l$, the resulting tensor linear system is over-determined; otherwise, it is under-determined. The initial guess is the zero tensor.

The number of iterations and the elapsed CPU time of the five methods are listed in Tables \ref{tab:overdetermined_results} and \ref{tab:underdetermined_results} for over-determined and under-determined cases respectively when the size of problems varies.

\begin{table}[!htbp]
\caption{Numerical results for over-determined tensors.}
\centering
\resizebox{\textwidth}{!}{
\begin{tabular}{cccccccccccccc}
\hline
$m$ & $l$ & $n$ & $p$
& \multicolumn{2}{c}{TRK}
& \multicolumn{2}{c}{TSKM}
& \multicolumn{2}{c}{TBEM}
& \multicolumn{2}{c}{RTK}
& \multicolumn{2}{c}{RTK-HB} \\
\cline{5-6} \cline{7-8} \cline{9-10} \cline{11-12} \cline{13-14}
& & &
& IT & CPU
& IT & CPU
& IT & CPU
& IT & CPU
& IT & CPU \\
\hline
200 & 20 & 10 & 20 & 650 & 0.2917 & 448 & 0.3261 & 86 & 0.1788 & 21 & 0.0591 & 16 & 0.0379 \\
500 & 20 & 10 & 20 & 550 & 0.2050 & 442 & 0.4366 & 64 & 0.2710 & 12 & 0.0744 & 10 & 0.0619 \\
500 & 40 & 20 & 20 & 1201 & 0.7604 & 958 & 1.9069 & 103 & 0.9194 & 18 & 0.2229 & 13 & 0.1593 \\
500 & 40 & 10 & 20 & 1212 & 0.5579 & 917 & 1.1011 & 93 & 0.4631 & 18 & 0.1552 & 14 & 0.1250 \\
500 & 20 & 10 & 40 & 562 & 0.2420 & 440 & 0.5181 & 66 & 0.4170 & 12 & 0.1128 & 10 & 0.0889 \\
800 & 20 & 10 & 20 & 555 & 0.1862 & 438 & 0.4386 & 33 & 0.2130 & 10 & 0.0867 & 9 & 0.0813 \\
\hline
\end{tabular}
}
\label{tab:overdetermined_results}
\end{table}

\begin{table}[!htbp]
\caption{Numerical results for under-determined tensors.}
\centering
\resizebox{\textwidth}{!}{
\begin{tabular}{cccccccccccccc}
\hline
$m$ & $l$ & $n$ & $p$
& \multicolumn{2}{c}{TRK}
& \multicolumn{2}{c}{TSKM}
& \multicolumn{2}{c}{TBEM}
& \multicolumn{2}{c}{RTK}
& \multicolumn{2}{c}{RTK-HB} \\
\cline{5-6} \cline{7-8} \cline{9-10} \cline{11-12} \cline{13-14}
& & &
& IT & CPU
& IT & CPU
& IT & CPU
& IT & CPU
& IT & CPU \\
\hline
20 & 200 & 10 & 20 & 492 & 0.5950 & 211 & 0.6145 & 110 & 0.2018 & 20 & 0.0322 & 16 & 0.0294 \\
20 & 500 & 10 & 20 & 354 & 1.2646 & 145 & 1.7831 & 75 & 0.4605 & 13 & 0.0781 & 11 & 0.0708 \\
40 & 500 & 20 & 20 & 1113 & 7.0888 & 383 & 6.7187 & 132 & 1.6891 & 18 & 0.2463 & 14 & 0.2050 \\
40 & 500 & 10 & 20 & 1057 & 3.6406 & 391 & 4.1208 & 147 & 0.8951 & 18 & 0.1347 & 14 & 0.1099 \\
20 & 500 & 10 & 40 & 348 & 2.1511 & 146 & 2.5844 & 82 & 0.7124 & 13 & 0.1146 & 11 & 0.1085 \\
20 & 800 & 10 & 20 & 328 & 1.9722 & 131 & 2.2981 & 64 & 0.5366 & 11 & 0.1206 & 9 & 0.1130 \\
\hline
\end{tabular}
}
\label{tab:underdetermined_results}
\end{table}

From Tables \ref{tab:overdetermined_results} and \ref{tab:underdetermined_results}, it is observed that both residual-based Kaczmarz and the accelerated variant are faster than the other three methods and the residual-based tensor Kaczmarz method with heavy ball momentum is always the fastest.

The curves of the relative error versus the number of iterations are plotted in Figs. \ref{fig:convergence1} and \ref{fig:convergence2} for the five methods, respectively.

\begin{figure}[!htbp]
\centering
\includegraphics[width=0.45\textwidth]{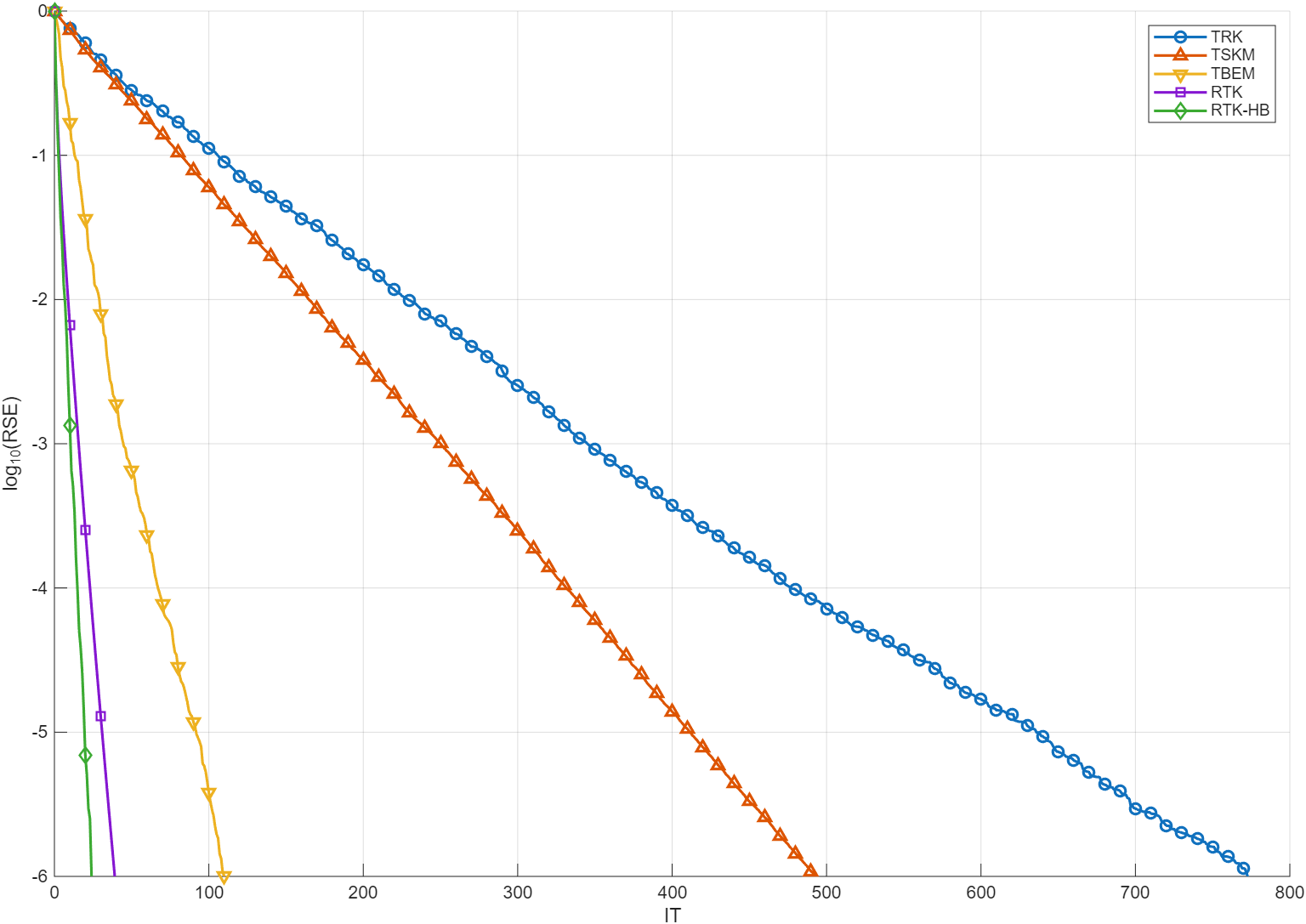}
\hspace{0.5cm}
\includegraphics[width=0.45\textwidth]{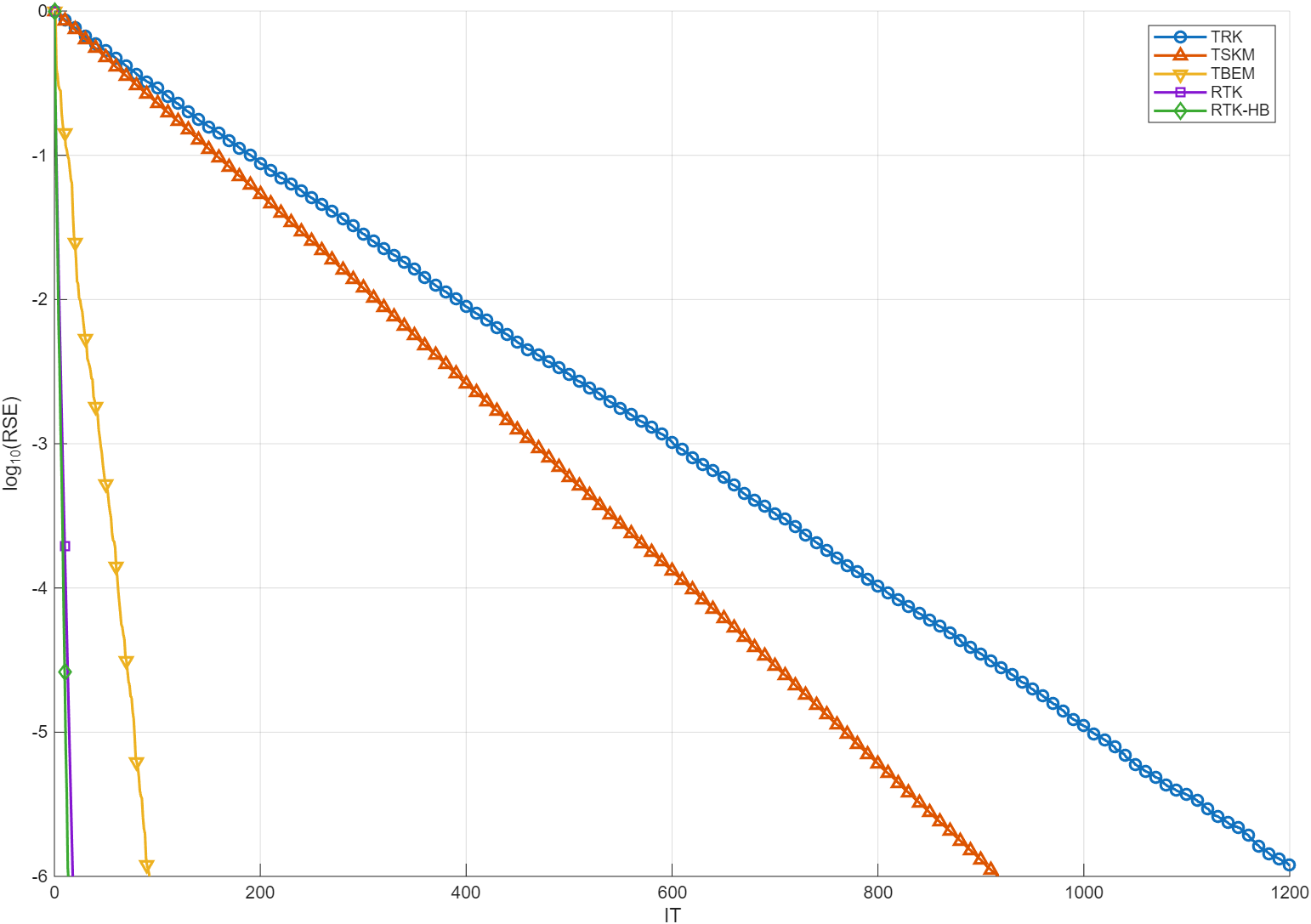}

\vspace{0.3cm}
(a) $100\times 20\times10$ \hspace{5cm} (b) $500\times 40\times10$

\vspace{0.3cm}
\caption{Convergence curves for over-determined tensors.} 
\label{fig:convergence1} 
\end{figure}

\begin{figure}[!htbp]
\centering
\includegraphics[width=0.45\textwidth]{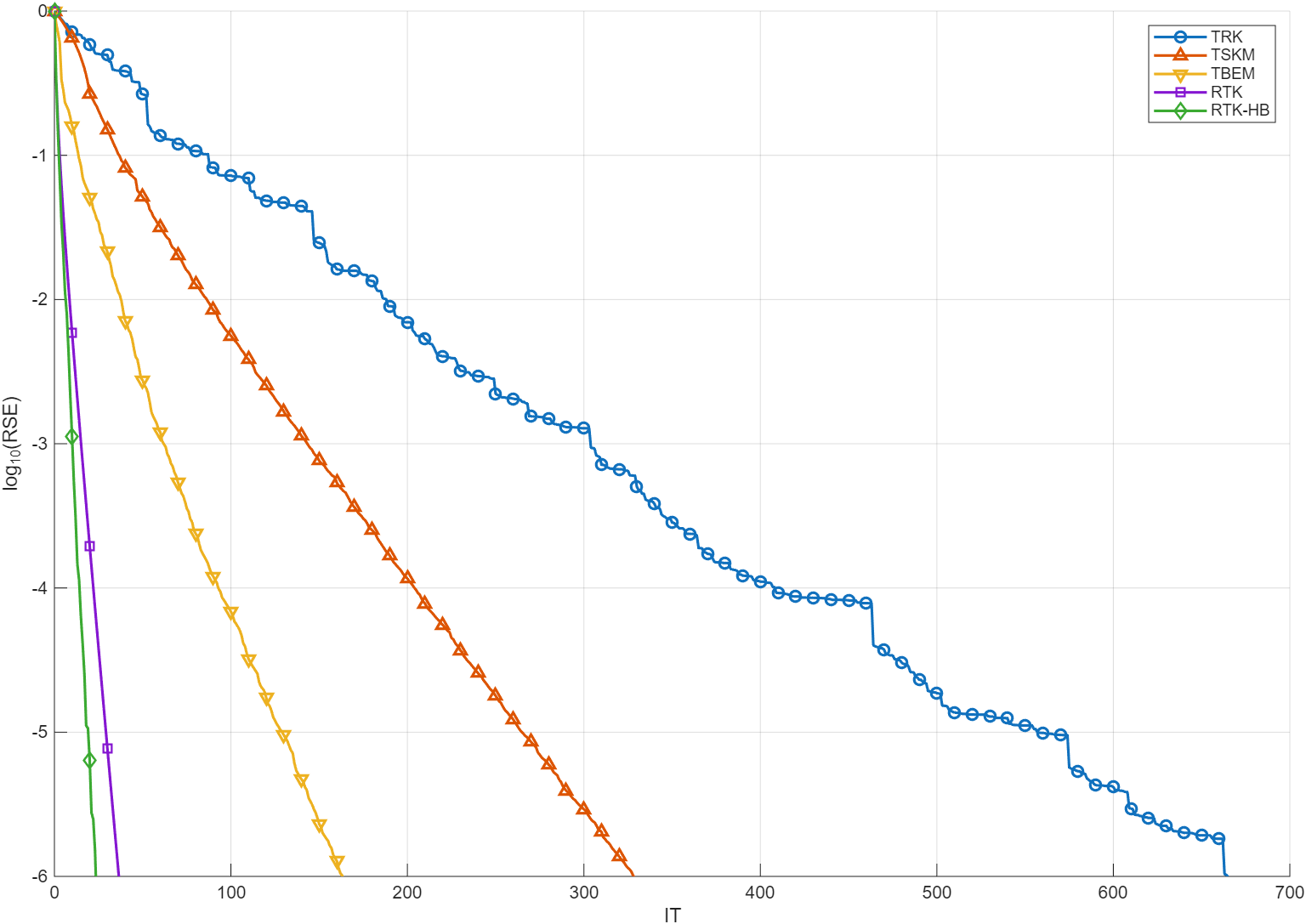}
\hspace{0.5cm}
\includegraphics[width=0.45\textwidth]{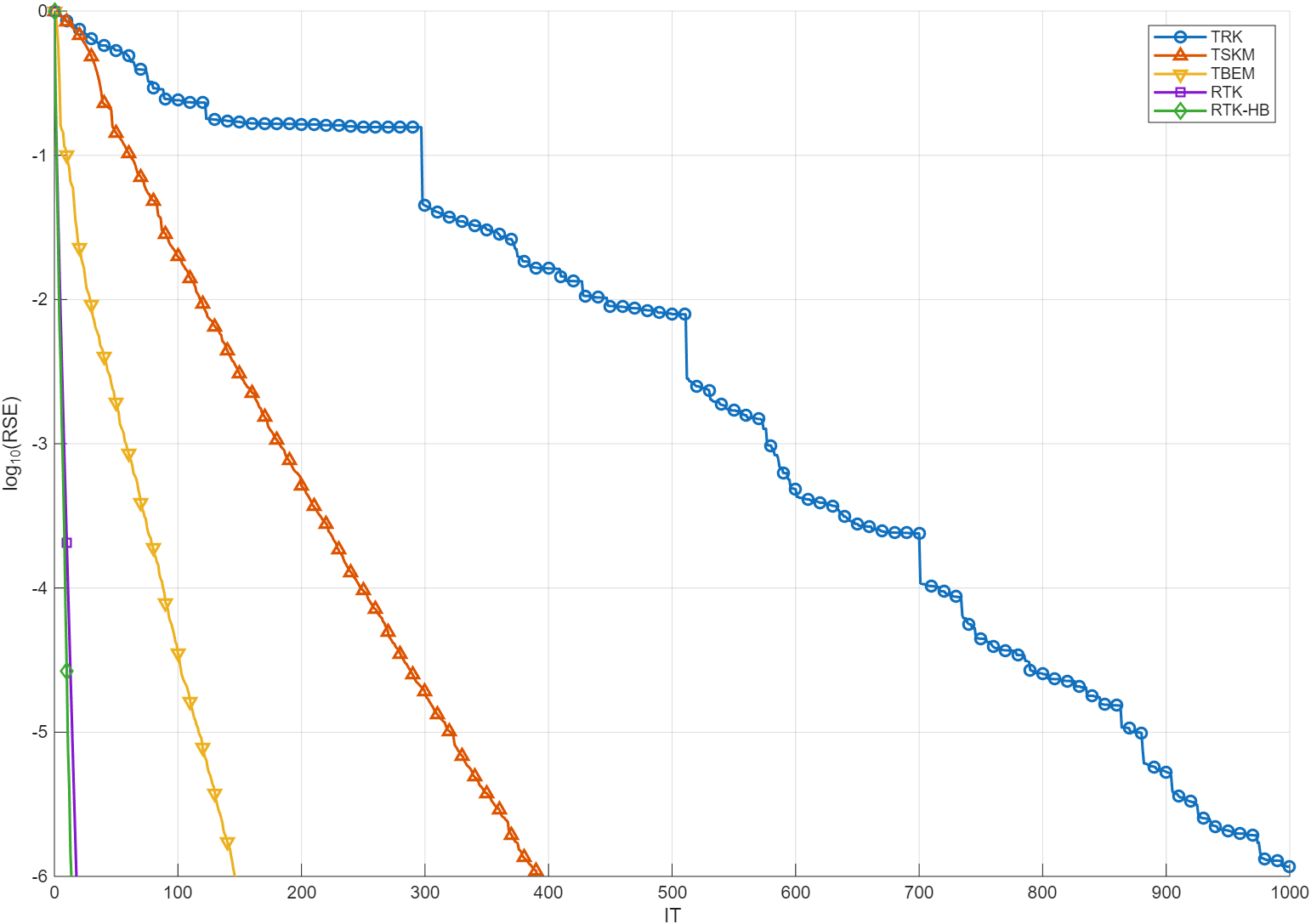}

\vspace{0.3cm}
(a) $20\times 100\times10$ \hspace{5cm} (b) $40\times 500\times10$

\vspace{0.3cm}
\caption{Convergence curves for under-determined tensors.} 
\label{fig:convergence2} 
\end{figure}

It is seen from Figs. \ref{fig:convergence1} and \ref{fig:convergence2} that the residual-based tensor Kaczmarz method converges significantly faster than the existing methods, while the residual-based tensor Kaczmarz method with heavy ball momentum exhibits even superior performance with the relative error decreasing most rapidly.
\end{example}

\begin{example}  \label{sy2}
In this experiment, the test tensors are generated with some images \cite{xu2026quaternion,WEI2026112609}. It evaluates the performance of the five methods on image deblurring. To fit the proposed tensor model, the data $\mathcal{X}$ of each image is reshaped into a third-order tensor of size $308 \times 1 \times 308$. Moreover, the measurement tensor $\mathcal{A} \in \mathbb{R}^{308 \times 308 \times 308}$ 
is constructed from the Gaussian blur kernel using a Toeplitz-based representation. The data of the blurred image is got by calculating $\mathcal{B}=\mathcal{A} *\mathcal{X}$ . The maximum number of iterations is set to 400 for the four images.

The image recovery results of five methods are presented in Fig. \ref{fig:image1}.

% \begin{figure}[!htbp]
%     \centering
%     \includegraphics[width=0.75\textwidth]{数值实验2/Figure_1.png}
    
%     % \vspace{0.3cm} 
%     \caption{The original clean image sequence and the recovered results of five methods.}
%     \label{fig:image1} 
% \end{figure}

\begin{figure}[!htbp]
    \centering

    \begin{minipage}{\textwidth}
        \begin{minipage}[c]{0.02\textwidth}
         \rotatebox{90}{Image 1}
        \end{minipage}
        \begin{minipage}[c]{0.9\textwidth}
            \centering
            \includegraphics[width=\textwidth]{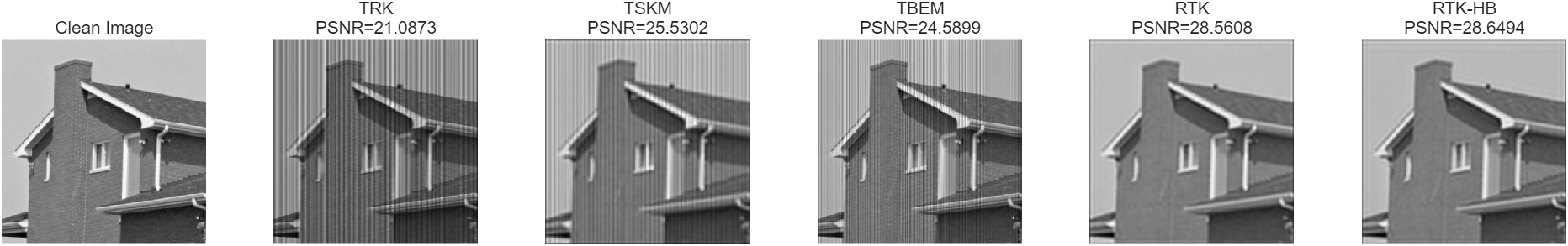}
        \end{minipage}
    \end{minipage}

    \vspace{0.3cm}

    \begin{minipage}{\textwidth}
        \begin{minipage}[c]{0.02\textwidth}
            \rotatebox{90}{Image 2}
        \end{minipage}
        \begin{minipage}[c]{0.9\textwidth}
            \centering
            \includegraphics[width=\textwidth]{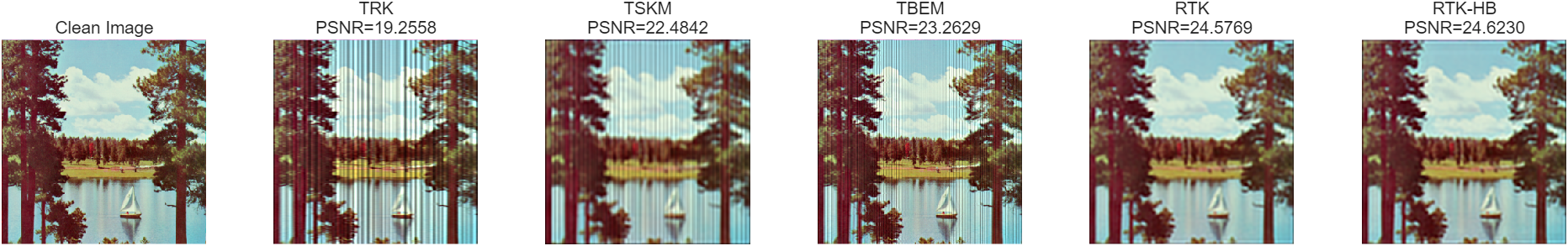}
        \end{minipage}
    \end{minipage}

    \vspace{0.3cm}

    \begin{minipage}{\textwidth}
        \begin{minipage}[c]{0.02\textwidth}
           \rotatebox{90}{Image 3}
        \end{minipage}
        \begin{minipage}[c]{0.9\textwidth}
            \centering
            \includegraphics[width=\textwidth]{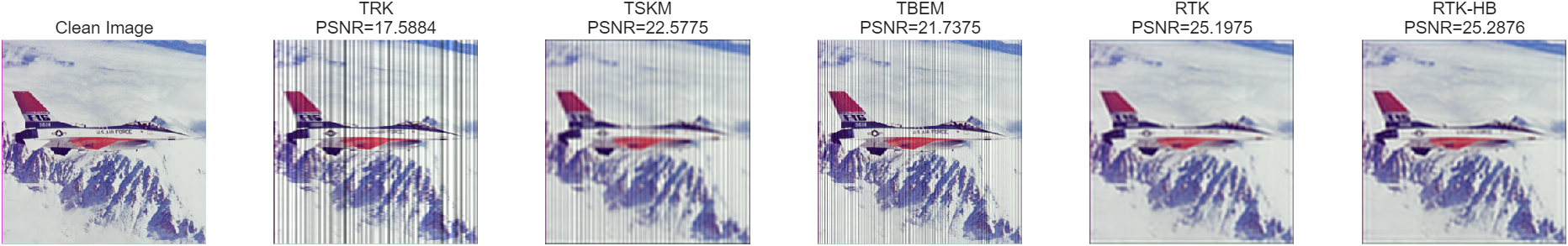}
        \end{minipage}
    \end{minipage}

    \vspace{0.3cm}

    \begin{minipage}{\textwidth}
        \begin{minipage}[c]{0.02\textwidth}
            \rotatebox{90}{Image 4}
        \end{minipage}
        \begin{minipage}[c]{0.9\textwidth}
            \centering
            \includegraphics[width=\textwidth]{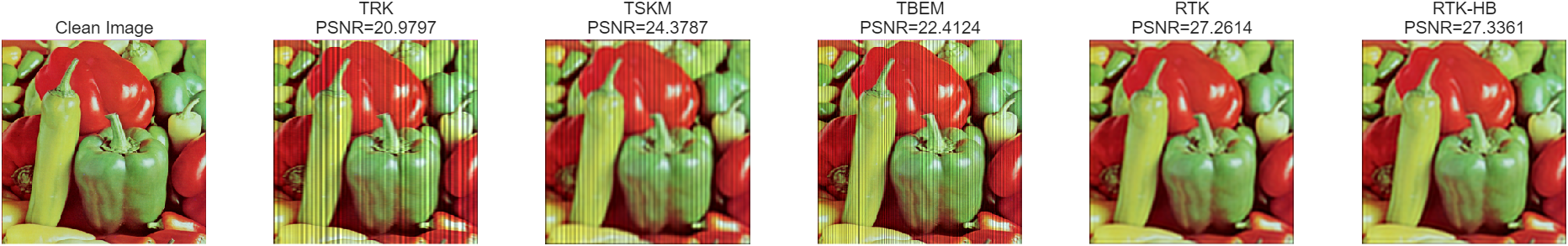}
        \end{minipage}
    \end{minipage}

    \caption{The original clean image sequence and the recovered results of five methods.}
    \label{fig:image1}
\end{figure}

It is seen from Fig. \ref{fig:image1} that the PSNR values of the two proposed methods are significantly higher than those of the other three methods.

The curves of the relative error versus the number of iterations and CPU time of Image 1 ``House'' are plotted in Fig. \ref{fig:convergence3} for the five methods, respectively.

\begin{figure}[!htbp]
\centering
\includegraphics[width=0.45\textwidth]{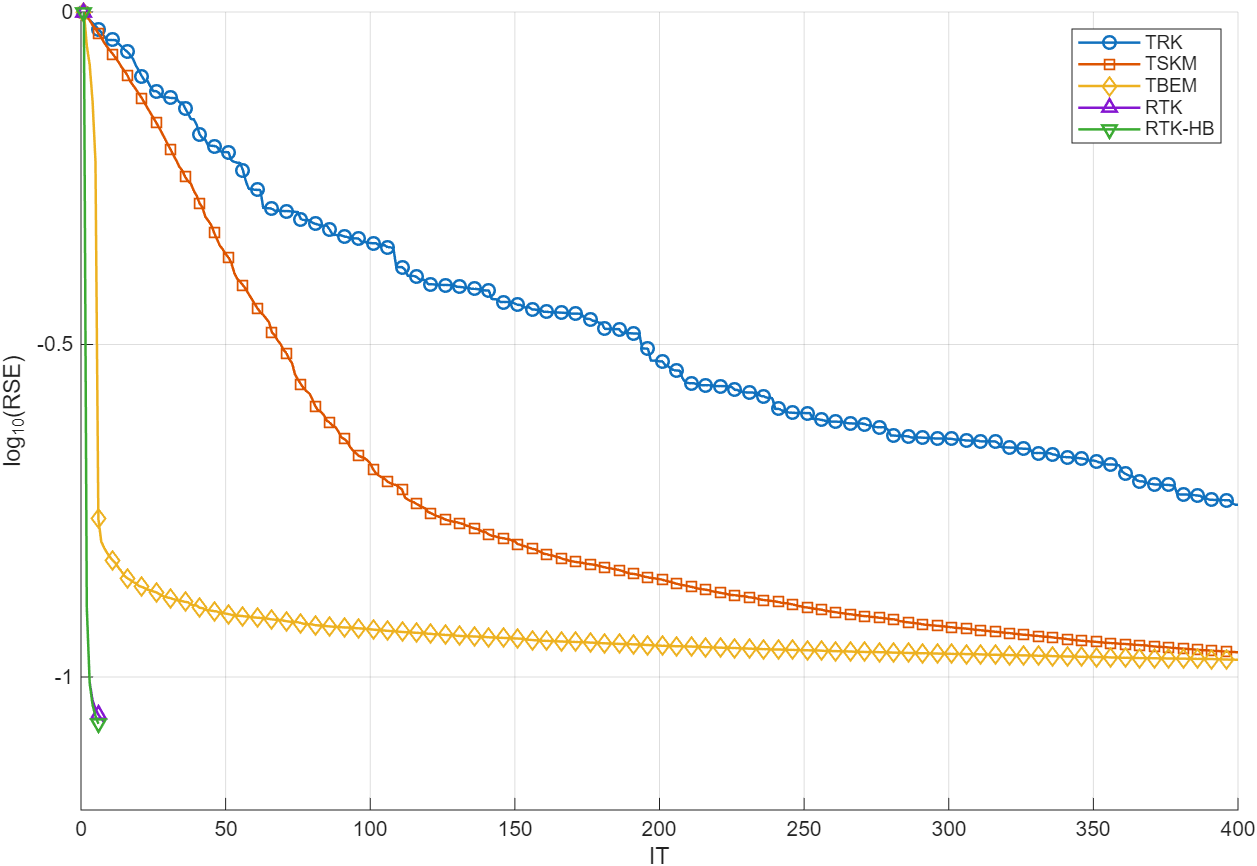}
\hspace{0.5cm}
\includegraphics[width=0.45\textwidth]{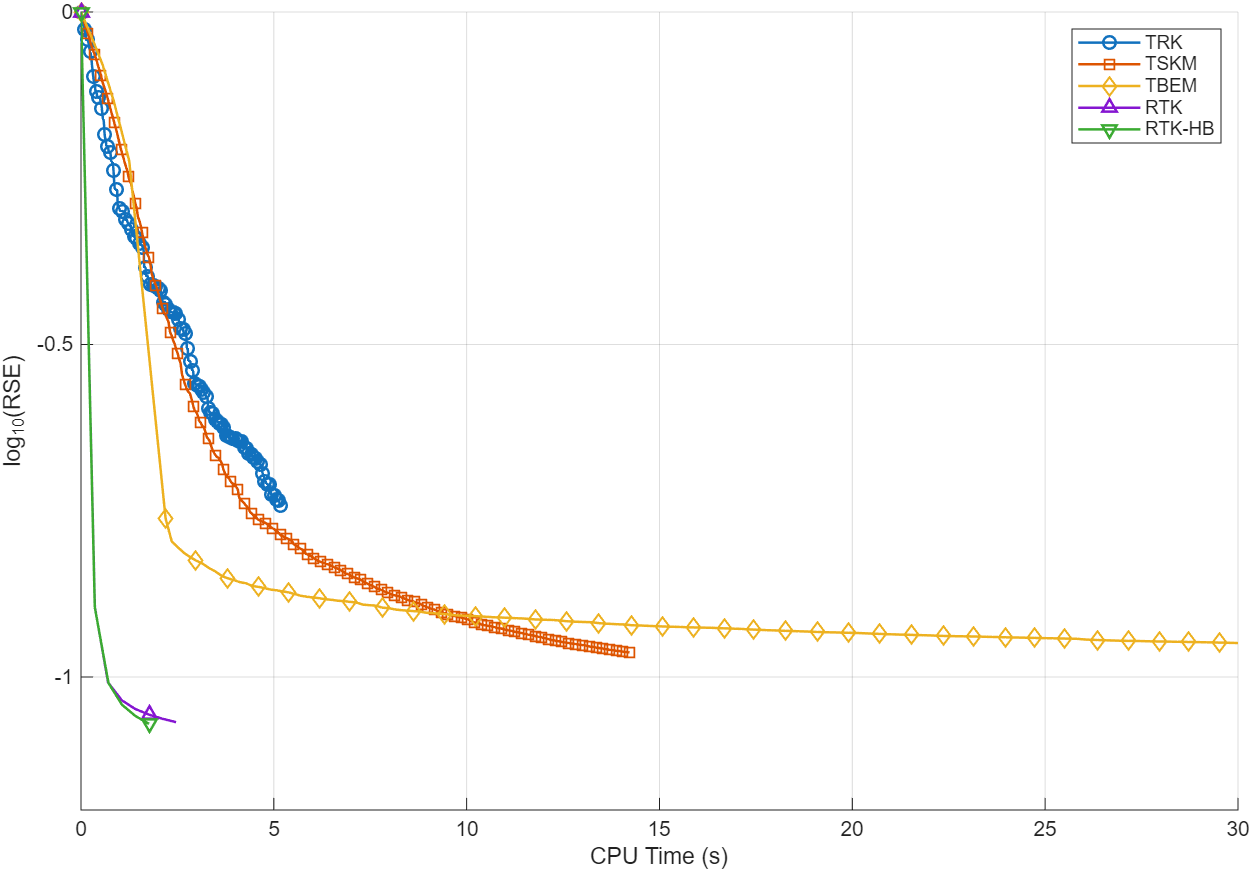}

% \vspace{0.5cm} 
\caption{Comparison of the relative error versus iterations and CPU time for five methods on the data of Image 1 ``House''.}
\label{fig:convergence3} 
\end{figure}

It is seen from Fig. \ref{fig:convergence3} that the two proposed methods show clear improvements over the other three methods. They both require fewer iterations and achieve a rapid decrease in the error within a very short CPU time.

\end{example}

\begin{example}In this experiment, the test tensor is generated from a hyperspectral image ``Urban''\cite{zou2025quaternion}. The original image $\mathcal{X}_{\text{ori}} \in \mathbb{R}^{307 \times 307 \times 162}$ is resized and reshaped into a tensor $\mathcal{X}^* \in \mathbb{R}^{308 \times 162 \times 308}$. The measurement tensor $\mathcal{A} \in \mathbb{R}^{308 \times 308 \times 308}$ and  $\mathcal{B}$ are established the same way as Example  \ref{sy2}. The maximum number of iterations is set to 250 and the iterations will stop when the relative error is less than 0.2.

The image recovery results selected from bands 110, 70 and 30 are presented in Fig. \ref{fig:gg}.

\begin{figure}[!htbp]
    \centering
    \includegraphics[width=0.75\textwidth]{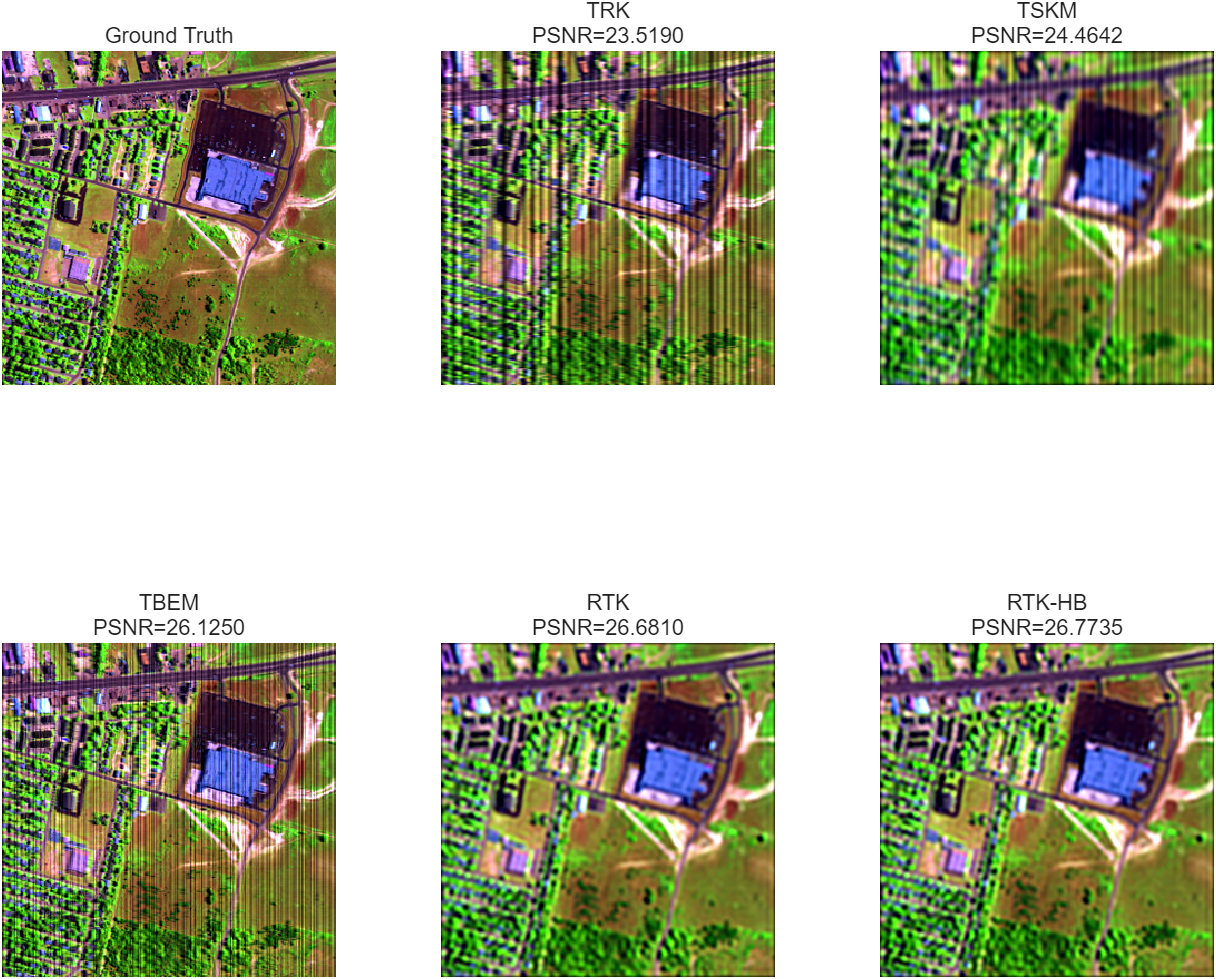}
    
    \caption{The original clean image sequence and the recovered results of five methods.}
    \label{fig:gg} 
\end{figure}

It is seen from Fig. \ref{fig:gg} that the residual-based tensor Kaczmarz method and the residual-based tensor Kaczmarz method with heavy ball momentum achieve superior reconstruction quality compared to the other three methods.

The curves of the relative error versus the number of iterations and CPU time are plotted in Fig. \ref{fig:convergence4} for five methods, respectively.

\begin{figure}[!htbp]
    \centering
    \includegraphics[width=0.45\textwidth]{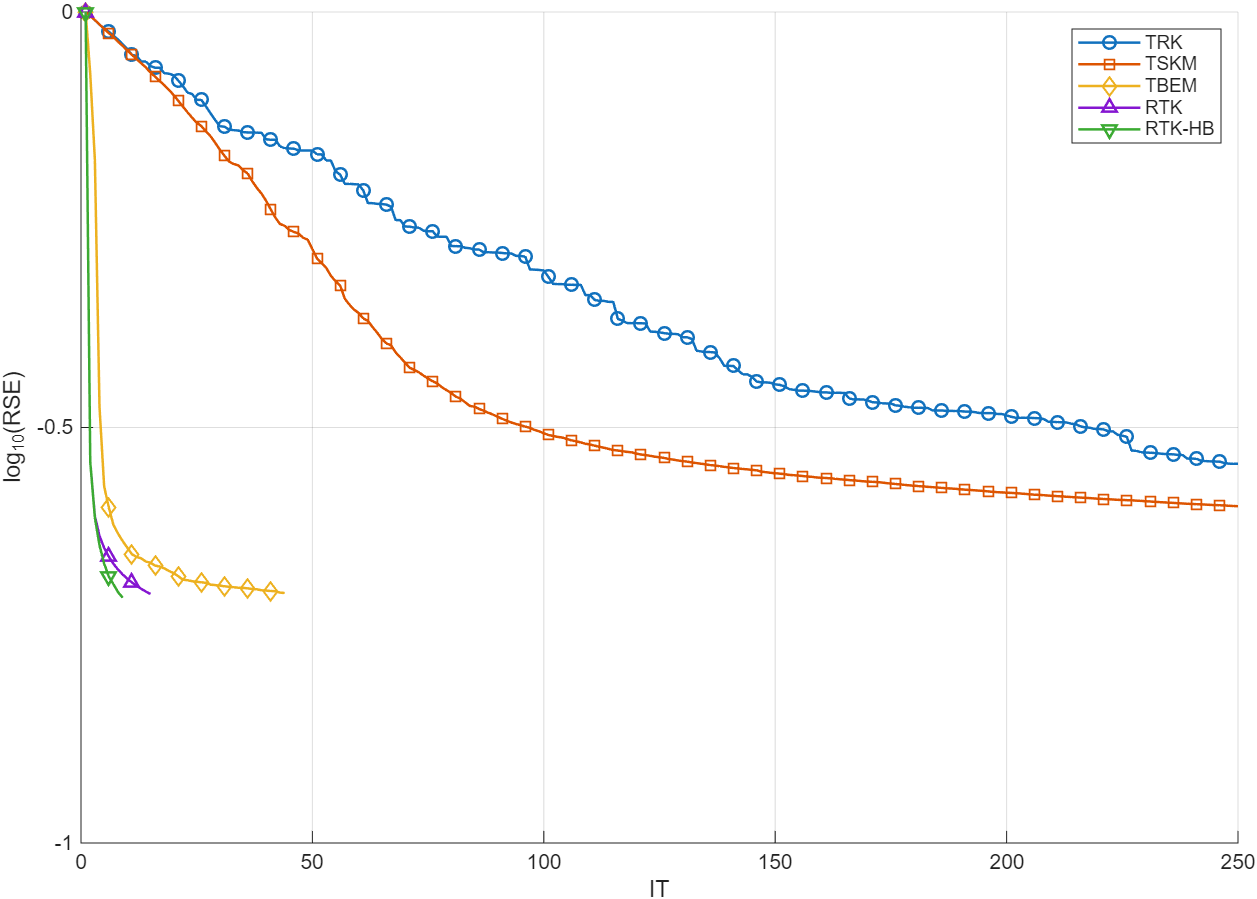}
    \hspace{0.5cm}
    \includegraphics[width=0.45\textwidth]{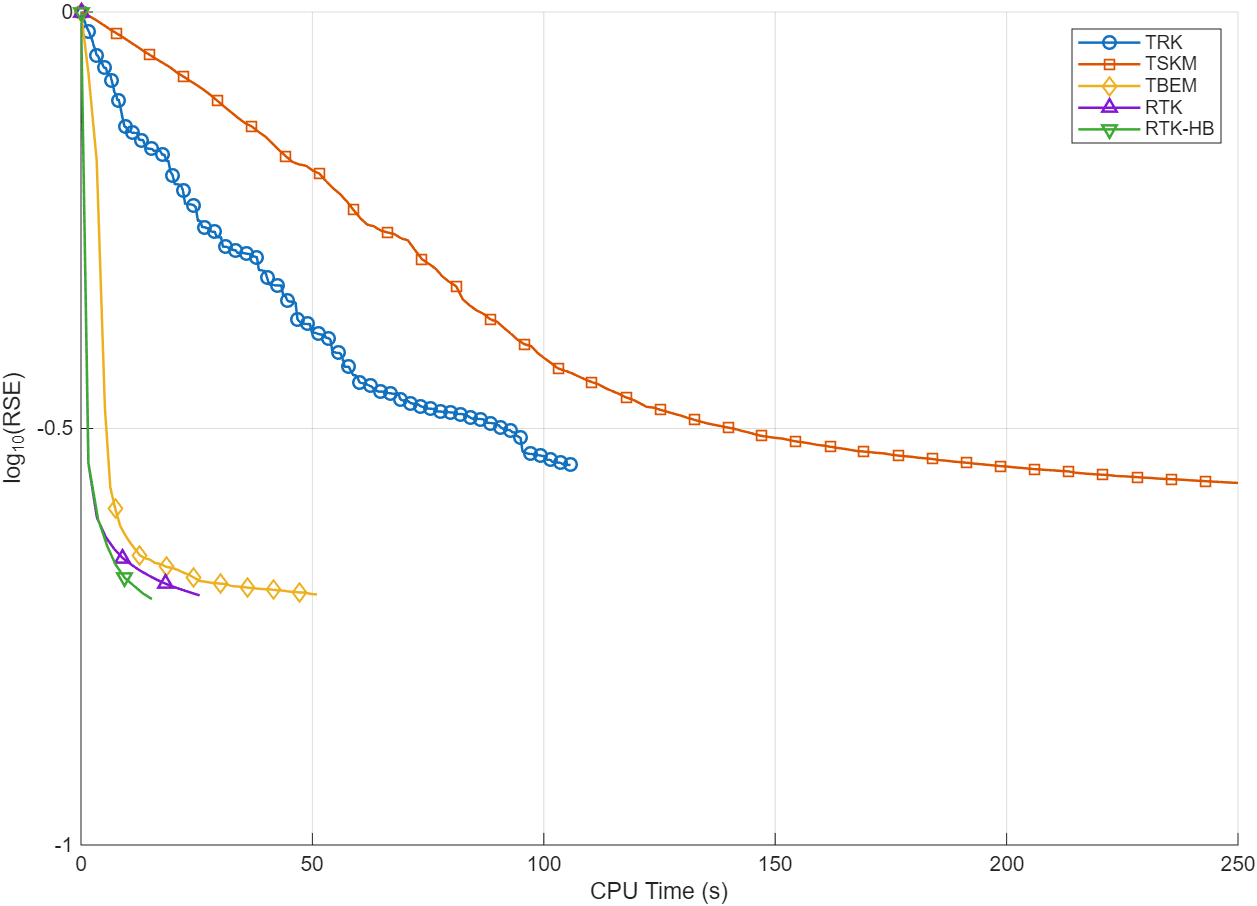}
   
    \caption{Comparison of the relative error versus iterations and CPU time for five methods on the hyperspectral image data.}
    \label{fig:convergence4}
\end{figure}

It is seen from Fig. \ref{fig:convergence4} that the proposed methods require fewer iterations and less computational time while achieving higher accuracy levels that the other methods fail to reach. 
%Among all the methods, the residual-based tensor Kaczmarz method with heavy ball momentum exhibits the best overall performance.
\end{example}

\section{Conclusion}   \label{sec5}

A residual-based Kaczmarz method and a heavy ball momentum-accelerated residual-based Kaczmarz method are proposed. Theoretical analysis established the convergence of the proposed methods and showed that they converge faster than several existing tensor Kaczmarz-type methods.
%The implementation of Fourier-transformed variant are given to
%further enhance the computational efficiency by FFT.
Numerical examples are given to further confirm the efficiency of our approaches
and demonstrate that the number of iterations and CPU time are significantly reduced compared to the existing tensor Kaczmarz method for large-scale tensor problems.

\vspace{1.5em} 

\noindent \textbf{Funding} This work was supported by National Natural Science Foundation of China (No. 12471357).

\vspace{1.5em} 

\noindent \textbf{Data Availability} The data that support the findings of this study are available upon reasonable request from the authors.

\section*{Declarations}
\noindent \textbf{Conflicts of Interest} The authors declare that they have no conflict of interest.

\bibliographystyle{plain}
\bibliography{main}

\end{document}